\documentclass[12pt]{article}
\usepackage{amsthm}
\usepackage{amsfonts}
\usepackage{epsfig}
\usepackage{cite}
\newtheorem{theorem}{Theorem}
\newtheorem{lemma}[theorem]{Lemma}

\newtheorem{proposition}[theorem]{Proposition}
\newtheorem{conjecture}{Conjecture}

\def\ca{\alpha}
\def\cb{\beta}
\def\cc{\gamma}
\def\cd{\delta}
\def\ce{\varepsilon}
\def\cf{\varphi}
\begin{document}
\title{Packing six $T$-joins in plane graphs}
\author{Zden{\v e}k Dvo{\v r}{\'a}k\thanks{
          Computer Science Institute, Faculty of Mathematics and
          Physics, Charles University, Malostransk{\'e} n{\'a}m{\v e}st{\'\i} 25,
          118 00 Prague, Czech Republic. E-mail: \texttt{rakdver@kam.mff.cuni.cz}.
          The Institute for Theoretical Computer Science (ITI) is supported by
          Ministry of Education of the Czech Republic as project 1M0545.
          This research has been supported by grant GACR 201/09/0197.}\and
        Ken-ichi Kawarabayashi\thanks{National Institute of Informatics, 2-1-2,
        Hitotsubashi, Chiyoda-ku, Tokyo 101-8430, Japan. E-mail: \texttt{k\_keniti@nii.ac.jp}.
        JST ERATO Kawarabayashi Large Graph project.
         Research partly supported by Japan Society for the Promotion of Science,
Grant-in-Aid for Scientific Research, by C\&C Foundation, by
Kayamori Foundation and by Inoue Research Award for Young
Scientists.}\and
    Daniel Kr{\'a}l'\thanks{
          Mathematics Institute, DIMAP and Department of Computer Science,
	  University of Warwick, Coventry CV4 7AL, UK.
	  E-mail: \texttt{d.kral@warwick.ac.uk}.
          Previous affiliation:
	  Institute for Theoretical Computer Science, Faculty of Mathematics and
          Physics, Charles University, Malostransk{\'e} n{\'a}m{\v e}st{\'\i} 25,
          118 00 Prague, Czech Republic.
          The Institute for Theoretical Computer Science (ITI) is supported by
          Ministry of Education of the Czech Republic as project 1M0545.
          This research has been supported by grant GACR 201/09/0197.
          A major revision of the paper was done during a visit of this author to the Institut Mittag-Leffler (Djursholm, Sweden).}}
\date{}
\maketitle
\begin{abstract}
Let $G$ be a plane graph and $T$ an even subset of its vertices.
It has been conjectured that if all $T$-cuts of $G$ have the same parity and
the size of every $T$-cut is at least $k$, then $G$ contains $k$ edge-disjoint $T$-joins.
The case $k=3$ is equivalent to the Four Color Theorem, and
the cases $k=4$, which was conjectured by Seymour, and $k=5$ were proved by Guenin.
We settle the next open case $k=6$.
\end{abstract}

\section{Introduction}

We study packings of $T$-joins in plane graphs.
Let $G$ be a graph and $T$ an even-size subset of its vertices.
A {\em $T$-join} is a subgraph $H$ of $G$ such that
the odd-degree vertices of $H$ are precisely those in $T$.
A {\em cut} is a partition of the vertex set of a graph $G$ into two sets $A$ and $B$,
which we refer to as {\em sides\/}; the {\em size} of the cut is the number of edges
with one end-vertex in $A$ and the other end-vertex in $B$.
A cut is {\em trivial} if one its sides consists of a single vertex and
a cut is {\em odd} if the size of $A$ is odd.
Finally, a {\em $T$-cut} is a cut such that $|T\cap A|$ is odd.

Clearly, if $G$ has a $T$-cut of size $k$,
it cannot have more than $k$ edge-disjoint $T$-joins.
We are interested when the converse is also true.
Seymour~\cite{bib-seymour} (also see Problem 12.18 in~\cite{bib-graphcol})
conjectured the following for $k=4$.

\begin{conjecture}
\label{conj}
Let $G$ be a plane graph and $T$ an even-size subset of its vertices.
If the sizes of all $T$-cuts in $G$ have the same parity and
the size of every $T$-cut is at least $k$, then $G$ contains $k$ edge-disjoint $T$-joins.
\end{conjecture}

The case $k=3$ is equivalent to the Four Color Theorem. The cases
$k=4$ and $k=5$ were proved by Guenin~\cite{bib-guenin}. We remark
that the case $k=4$ implies the Four Color Theorem, as pointed out
by Seymour~\cite{bib-seymour}. Here, we prove the next open case.
Our main result is the following.

\begin{theorem}
\label{thm-main}
Let $G$ be a plane multigraph and $T$ an even subset of its vertices.
If every $T$-cut of $G$ has the same parity and
the size of every $T$-cut is at least six,
then $G$ contains six edge-disjoint $T$-joins.
\end{theorem}
We note that the cases $k=7$ and $k=8$ of Conjecture~\ref{conj} have recently
been proven in~\cite{bib-tjoin7+,bib-tjoin8}.

Guenin~\cite{bib-guenin} argued that it suffices to prove Conjecture~\ref{conj}
for plane graphs $G$ with $V(G)=T$ that are $k$-regular, i.e., every vertex has degree $k$.
In such graphs, the existence of $k$ edge-disjoint $T$-joins is equivalent to the existence
of a $k$-edge-coloring; a {\em $k$-edge-coloring} is an assignment of $k$ colors
to the edges such that no vertex is incident with two edges of the same color.
Hence, Theorem~\ref{thm-main} for $k=6$ is equivalent to the next theorem
which we prove in the following sections of the paper.

\begin{theorem}
\label{thm-edge-coloring}
Let $G$ be a $6$-regular plane multigraph.
If $G$ has no odd cut of size less than $6$,
then $G$ has a $6$-edge-coloring.
\end{theorem}

Note that the condition that $G$ has no odd cut of size less than six
implies that the number of vertices of $G$ is even (otherwise, consider
a cut with one of the sides empty).
Let us remark that Conjecture~\ref{conj} would be implied by the following
more general conjecture of Seymour (replacing the condition of not containing
Petersen by a stronger condition of being planar yields a statement equivalent
to Conjecture~\ref{conj}).

\begin{conjecture}
\label{conj-petersen}
Let $G$ be a $k$-regular graph with no Petersen minor.
The graph $G$ is $k$-edge-colorable if and only if every odd cut of $G$ has size at least $k$.
\end{conjecture}

The case $k=3$ is Tutte's well-known three-edge-coloring
conjecture, whose solution has been announced by Robertson, Sanders,
Seymour and Thomas (see \cite{rst}). Indeed, the case $k=3$ is a
special case of another well known conjecture by Tutte, which is
known as Tutte's four flow conjecture.

Conjecture~\ref{conj-petersen} would also imply the following conjecture of Conforti and Johnson~\cite{bib-techrep},
also see~\cite{bib-cornuejols}.

\begin{conjecture}
\label{conj-postman}
Let $G$ be a graph with no Petersen minor and $T$ a set of its odd-degree vertices.
Then, the maximum number of edge-disjoint $T$-joins is equal to the size of the smallest $T$-cut.
\end{conjecture}

Conjectures \ref{conj}, \ref{conj-petersen} and \ref{conj-postman}
have attracted attention of many researchers, because they are connected
not only to $T$-joins, $T$-cuts and edge-coloring but also to cycle
covers and flows. For more details, we refer the reader to the books
by Cornu{\'e}jols \cite{bib-cornuejols} and by Schrijver \cite{lex},
respectively.

\section{Notation}

We now introduce notation used throughout the paper.
Since any graph satisfying the assumptions of Theorem~\ref{thm-edge-coloring} is $2$-connected,
the introduced notation will be used only for $2$-connected graphs.
A vertex of degree $d$ is called a {\em $d$-vertex}.
An {\em $\ge d$-vertex} is a vertex of degree at least $d$ and
an {\em $\le d$-vertex} is a vertex of degree at most $d$.
In a $2$-connected plane graph, a $d$-face is a face incident
with exactly $d$ edges. Analogously to vertices,
we use an $\le d$-face and an $\ge d$-face.

A {\em bigon} is a $2$-face and
a {\em multigon} is a maximal sequence of bigons such that every pair of consecutive bigons share an edge.
The {\em order} of a multigon is the number of edges forming it,
i.e., the number of bigons forming it increased by one.
Multigons of order three are called {\em trigons} and those of order four {\em quadragons}.
Two multigons are {\em incident} if they share a vertex.
If $f$ is a face, then two multigons are {\em $f$-incident} if they contain
edges consecutive on the boundary of $f$.
Similarly, a multigon and a face are $f$-incident if they contain edges consecutive on the boundary of $f$.
A multigon and a face are {\em adjacent} if they contain the same edge;
similarly, two faces are {\em adjacent} if they contain the same edge.

We say that a face is {\em $k$-big} if it is adjacent to exactly $k$ different $\ge 4$-faces.
A face is {\em $\le\ell$-big} if it is $k$-big for $k\le\ell$. We use {\em $\ge\ell$-big} in the analogous way.
A $5$-face $f$ is {\em dangerous}
if $f$ is adjacent to two trigons $f$-incident with the same bigon and
$f$ is adjacent to no other multigons.
Finally, a trigon $t$ is {\em dangerous} if $t$ is adjacent to a dangerous $5$-face.

The rest of the paper is devoted to proving Theorem~\ref{thm-edge-coloring}.
With respect to this proof,
a plane graph $G$ is said to be a {\em minimal counterexample}
if $G$ satisfies the assumptions of Theorem~\ref{thm-edge-coloring}, i.e.,
\begin{itemize}
\item $G$ is $6$-regular, and
\item every odd cut of $G$ has size at least six (note that these two properties imply that $G$ is $2$-connected), and
\end{itemize}
$G$ has no $6$-edge-coloring, and it also holds that
\begin{itemize}
\item subject to the previous conditions, $G$ has the smallest number of vertices,
\item subject to the previous conditions, $G$ has as many quadragons as possible,
\item subject to the previous conditions, $G$ has as many trigons as possible, and
\item subject to the previous conditions, $G$ has as many bigons as possible.
\end{itemize}
We exclude the existence of a minimal counterexample which proves Theorem~\ref{thm-edge-coloring}.

In our arguments, we will often need to transform an edge-coloring to another one.
To simplify our arguments, we will use the letters $\ca$, $\cb$, $\cc$, $\cd$, $\ce$ and
$\cf$ to denote the colors used on edges.
If $G$ is a graph with maximum degree $d$ that is $d$-edge-colored,
then an {\em $\ca\cb$-chain}, where $\ca$ and $\cb$ are two colors used on the edges of $G$,
is a cycle or a maximal path formed by edges with the colors $\ca$ and $\cb$ only.
{\em Swapping} the colors the of edges on an $\ca\cb$-chain means recoloring
$\ca$-colored edges of the chain with $\cb$ and $\cb$-colored edges with $\ca$.

\section{Structure of a minimal counterexample}
\label{sect-structure}

In this section, we analyze the structure of a minimal counterexample; in the next section,
we then prove that there exists no minimal counterexample using the discharging method.

\subsection{Odd cuts}

We start with analyzing sizes and the structure of odd cuts in a minimal counterexample.
As the first step, we prove the following simple observation.

\begin{lemma}
\label{lm-odd-cut}
Every non-trivial odd cut in a minimal counterexample $G$ has size at least eight.
\end{lemma}

\begin{proof}
Since $G$ is $6$-regular, every cut in $G$ has even size. Hence, if $G$ has a non-trivial
odd cut of size less than eight, its size must be six.
Let $A$ and $B$ be the sides of such a non-trivial odd cut.

Let $G_A$ be the (plane) graph obtained from $G$ by replacing $A$ with a single vertex
incident with the six edges of the cut $(A,B)$. Similarly, $G_B$ is the graph obtained
from $G$ by replacing $B$ with a single vertex incident with the six edges of the cut $(A,B)$.
By the minimality of $G$, both $G_A$ and $G_B$ have $6$-edge-colorings. These edge-colorings
combine to a $6$-edge-coloring of $G$ (the edges of the cut receive six distinct colors)
which contradicts that $G$ is a minimal counterexample.
\end{proof}

Using Lemma~\ref{lm-odd-cut},
we prove the following simple observation on the structure of multigons in a minimal counterexample.

\begin{lemma}
\label{lm-multigon}
In a minimal counterexample $G$,
the order of every multigon is at most four and
the sum of the orders of any two incident multigons is at most five.
\end{lemma}

\begin{proof}
If $G$ contains a multigon of order five or
two incident multigons with orders summing to six,
then there is a vertex $v$ that has only two neighbors,
say $v'$ and $v''$. Unless $G$ has exactly four vertices,
the set $\{v,v',v''\}$ forms a side of a non-trivial odd cut of size six
which is impossible by Lemma~\ref{lm-odd-cut}.
Hence, $G$ has exactly four vertices and
it is straightforward to show that it can be $6$-edge-colored.
\end{proof}

A possible way to obtain an edge-coloring of a minimal counterexample
is reducing a minimal counterexample
to another $6$-regular graph of the same order but with multigons of larger orders.
An operation that will achieve this is the swapping operation that we now introduce;
this operation can only be performed
if the resulting graph has no odd cuts of size less than six.

If $G$ is a plane graph such that it is possible to draw a closed curve in the plane that
intersects $G$ only at vertices $v_1,\ldots,v_k$ for $k$ even and
$G$ contains edges $v_2v_3$, $v_4v_5$, $\ldots$, $v_{k-2,k-1}$ and $v_kv_1$ (such
a $k$-tuple of vertices $v_1,\ldots,v_k$ is called {\em eligible}),
then the graph obtained from $G$ by the {\em $v_1\ldots v_k$-swap}
is the plane graph obtained by removing the edges $v_2v_3$, $v_4v_5$, $\ldots$, $v_{k-2,k-1}$ and $v_kv_1$ and
inserting the edges $v_{2i-1}v_{2i}$ for $i=1,\ldots,k/2$.

A crucial property of this operation is that the size of odd cuts can decrease
by at most two if $k$ is four or six. We prove this in the next lemma.

\begin{lemma}
\label{lm-swap}
Let $G$ be a minimal counterexample and let $k$ be either four or six.
Any graph $G'$ obtained from $G$ by the $v_1v_2\cdots v_k$-swap for eligible vertices $v_1,\ldots,v_k$
is $6$-regular and has no odd cut of size less than six.
\end{lemma}

\begin{proof}
Clearly, the graph $G'$ is $6$-regular.
Consider a non-trivial odd cut with sides $A$ and $B$ of $G'$.
Observe that this cut is also odd in $G$.
By symmetry, we can assume that $|A\cap\{v_1,\ldots,v_k\}|\le 3$.
Let $A'\subseteq A$ be those vertices $v_i$ of $A$ such that $v_{i-1}$ or $v_{i+1}$ is in $A$.
Since $|A\cap\{v_1,\ldots,v_k\}|\le 3$, the set $A'$ is either empty or
formed by two or three consecutive vertices.
If $A'$ is empty or formed by three consecutive vertices,
the number of edges leaving $A'$ to $B$ is the same in $G$ and $G'$.
If $A$ is formed by two consecutive vertices,
then the number of such edges leaving $A'$ to $B$ is either increased or decreased by two.
Since the number of edges between the vertices of $A\setminus A'$ and the vertices of $B$ is preserved,
the size of the cut with sides $A$ and $B$ in $G'$ differ by at most two from its size in $G$.
Since $G$ has no non-trivial odd cuts of size less than eight by Lemma~\ref{lm-odd-cut},
the lemma follows.
\end{proof}

\subsection{Existence of $e$-colorings}

A crucial property of a minimal counterexample
is the existence of an $e$-coloring. Let us define this notion formally.
If $G$ is a $6$-regular graph and $e$ an edge of $G$, then an {\em $e$-coloring} of
$G$ is a coloring of edges of $G$ with six colors such that every edge except for $e$
is assigned one color, $e$ is assigned three colors, and
for every color, each vertex is incident with an odd number of edges
assigned that color. Observe that in an $e$-coloring, every vertex except the end-vertices of $e$
must be incident with edges of pairwise distinct colors, i.e., an $e$-coloring
is proper at all vertices except for the end vertices of $e$.

The following lemma appears as Lemma 2.5 in~\cite{bib-guenin}.

\begin{lemma}
\label{lm-e-coloring}
Let $G$ be a minimal counterexample. For every edge $e$, there exists an $e$-coloring.
\end{lemma}

We now strengthen Lemma~\ref{lm-e-coloring} for the case
when $e$ is contained in a multigon of order at least three.

\begin{lemma}
\label{lm-e-trigon}
Let $G$ be a minimal counterexample, $e=vv'$ an edge of $G$ contained in a multigon, and
$w$ and $w'$ neighbors of $v$ and $v'$, respectively, such that $vv'w'w$ is eligible.
Suppose that either $e$ is contained in a multigon of order at least three or
$e$ is contained in a bigon and neither $vv'$ nor $ww'$ is contained in a multigon of order at least three.

There exists an $e$-coloring such that $e$ is assigned precisely three colors, say $\ca$, $\cb$ and $\cc$, and
one of these colors, say $\ca$, is assigned to two other edges incident with $v$ including $vw$ as well as
to two other edges incident with $v'$ including $v'w'$.
Each of the vertices $v$ and $v'$ is incident with a single edge of each color except for the color $\ca$.
Moreover, if $e$ is contained in a bigon, the other edge of the bigon is colored with $\cd$, and
if $e$ is contained in a multigon of order three or more, two of the other edges of the multigon are colored with $\cd$ and $\ce$.
\end{lemma}

\begin{proof}
Let $G'$ be the graph resulting by the $vv'w'w$-swap.
By the minimality of $G$, $G'$ has a $6$-edge-coloring (this follows from the assumption that
either $e$ is contained in a multigon of order at least three or
$e$ is contained in a bigon and neither $vv'$ nor $ww'$ is contained in a multigon of order at least three).
Let $\ca$ be the color of the new edge $ww'$.
Note that $\ca$ is not assigned to any edge forming the multigon containing $e$ (otherwise, $G$ has a $6$-edge-coloring).
Let $\cb$ be the color of $e$ and $\cc$ the color of the new edge $vv'$.
The desired coloring of $G$ is obtained
by removing the new edges $vv'$ and $ww'$,
inserting edges $vw$ and $v'w'$ colored with $\ca$,
assigning the colors $\ca$, $\cb$ and $\cc$ to $e$, and
permuting the colors $\cd$, $\ce$ and $\cf$ to satisfy the last statement of the lemma.
\end{proof}

The notion of $e$-colorings is related to the notion of mates which also appeared in~\cite{bib-guenin}.
Here, we use a slightly different but equivalent terminology to that in~\cite{bib-guenin}.
If $G$ is a minimal counterexample, $e$ is an edge of $G$ and
$c$ is one of the colors used in an $e$-coloring, then a {\em $c$-mate} $M_c$
is a set of edges of $G$ that form an non-trivial odd cut containing $e$ such
that, for every color $c'\not=c$,
$M_c$ contains exactly one edge that is assigned the color $c'$.
Lemma~2.6 in~\cite{bib-guenin} (note that our definition of a minimal counterexample and an $e$-coloring matches the setting in that Lemma~2.6 in~\cite{bib-guenin} is proven) asserts the existence of mates in a minimal counterexample.

\begin{lemma}
\label{lm-mate}
Let $G$ be a minimal counterexample and $e$ an edge of $G$.
For every $e$-coloring and every color $c$, there exists a $c$-mate.
\end{lemma}

The following observation on the structure of mates is often used in our arguments.
We state it as a proposition for future reference.

\begin{proposition}
\label{prop-mate}
Let $G$ be a minimal counterexample and $e$ an edge of $G$.
In each $e$-coloring,
every $c$-mate $M_c$ contains at least five edges (possibly including $e$)
assigned the color $c$.
\end{proposition}

\begin{proof}
By Lemma~\ref{lm-odd-cut}, the mate $M_c$ contains at least eight edges.
Since the edge $e$ is assigned at least three colors and $M_c$ contains exactly
one edge assigned each of the colors $c'\not=c$, $M_c$ must include at least five edges
assigned the color $c$.
\end{proof}

Let us demonstrate the use of Proposition~\ref{prop-mate} in the following lemma.

\begin{lemma}
\label{lm-face-quadragon}
In a minimal counterexample $G$,
every face $f$ adjacent to a quadragon $q$ is adjacent to at least five $\ge 4$-faces that are not $f$-incident with $q$.
In particular, $f$ is $\ge 5$-big $\ge 8$-face.
\end{lemma}

\begin{proof}
Let $e=vv'$ be the edge of the quadragon incident with $f$ and $wvv'w'$ a facial walk of $f$.
Consider an $e$-coloring as described in Lemma~\ref{lm-e-trigon}.
Let $M_c$ be a $c$-mate for $c\not=\ca$ and $e_c$ another edge of $f$ contained in $M_c$.
Let $f_c$ be the other face containing the edge $e_c$.
Since the edges of the quadragon are assigned all six colors and
the mate $M_c$ contains at least five edges with the color $c$ by Proposition~\ref{prop-mate},
the face $f_c$ contains another edge with the color $c$.
Hence, $f_c$ is $\ge 4$-face since no two edges with the color $c$ are incident.
Since the edge $e_c$ is not contained in a multigon and $G$ has no multigons of order five by Lemma~\ref{lm-multigon},
the faces $f_c$ differ for different choices of $c$.
We conclude that $f$ is $\ge 5$-big $\ge 8$-face.
\end{proof}

Another application of the mates is the following.

\begin{lemma}
\label{lm-face-trigon}
In a minimal counterexample,
every face $f$ adjacent to a trigon $t$ is incident with at least one edge that is not contained in a multigon and that is not $f$-incident with $t$.
\end{lemma}

\begin{proof}
Let $e=vv'$ be an edge of $t$ and $wvv'w'$ a facial walk of $f$.
Consider an $e$-coloring as described in Lemma~\ref{lm-e-trigon}.
Let $M_\cf$ be a $\cf$-mate. Since the edges of the trigon have all the five colors different from $\cf$,
all the other edges contained in $M_\cf$ have the color $\cf$.
Consequently, the edge of $f$ contained in $M_\cf$ and not in the trigon
is not contained in a multigon and it is also not $f$-incident with $t$.
\end{proof}

In the rest of this section, we will focus in more detail on multigons
adjacent to faces of various sizes.

\subsection{Structure of $3$-faces}

In this subsection, we focus on $3$-faces.
We start with $3$-faces adjacent to a trigon.

\begin{lemma}
\label{lm-3-face-trigon}
If a $3$-face $f$ of a minimal counterexample is adjacent to a trigon,
then $f$ is adjacent to no other multigon, both the other faces adjacent to $f$ are $\ge 5$-big and
the other face adjacent to the trigon is also $\ge 5$-big.
\end{lemma}

\begin{proof}
Let $e=v_1v_2$ be an edge of the trigon adjacent to $f$,
$v_3$ be the remaining vertex of the face $f$,
$f'$ the other face adjacent to the trigon,
$f''$ the face incident with the edge $v_1v_3$ and
$v_4$ a neighbor of $v_3$ on the face $f''$ (see Figure~\ref{fig-3-face-trigon}).
Consider an $e$-coloring as described in Lemma~\ref{lm-e-trigon} for the eligible sequence $v_1v_2v_3v_4$.

\begin{figure}
\begin{center}
\epsfbox{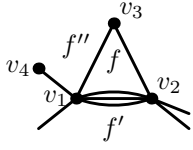}
\end{center}
\caption{Notation used in the proof of Lemma~\ref{lm-3-face-trigon}.}
\label{fig-3-face-trigon}
\end{figure}

Let $M_c$ be a $c$-mate for $c\not=\ca$.
Observe that $M_c$ contains the three edges of the trigon and the edge $v_1v_3$
which is colored with $\cf$ (it cannot contain the edge $v_2v_3$ because its color is $\ca$).
We show that both $f'$ and $f''$ are $\ge 5$-big.
Since $G$ is $2$-connected,
the faces $f'$ and $f''$ are different.
We now argue that $f'$ is $\ge 5$-big.
Let $e_c$ be the edge incident with $f'$ that is contained in $M_c$ and
that is not contained in the trigon. Clearly, the color of $e_c$ is $c$.
Let $f_c$ be the face adjacent to $e_c$ that is distinct from $f'$.
Since the mate $M_c$ contains at least five edges with the color $c$ by Proposition~\ref{prop-mate},
$f_c$ must be $\ge 4$-face.
Since the faces $f_c$ are different for different choices of $c\not=\ca$,
the face $f'$ is $\ge 5$-big.
The argument that $f''$ is $\ge 5$-big follows the same lines.

Switching the roles of $v_1$ and $v_2$ yields that
the face adjacent to $f$ distinct from $f''$ is also $\ge 5$-big.
\end{proof}

We now focus on $3$-faces adjacent to bigons.

\begin{lemma}
\label{lm-3-face-23-bigon}
If a $3$-face of a minimal counterexample is adjacent to at least two bigons,
then it is adjacent exactly to two bigons and the other faces adjacent to these
bigons are $\ge 5$-big.
\end{lemma}

\begin{figure}
\begin{center}
\epsfbox{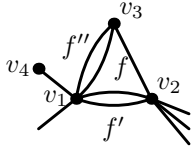}
\end{center}
\caption{Notation used in the proof of Lemma~\ref{lm-3-face-23-bigon}.}
\label{fig-3-face-23-bigon}
\end{figure}

\begin{proof}
By Lemma~\ref{lm-3-face-trigon}, $f$ is not adjacent to a trigon, and
since $G$ has no non-trivial odd cut of size six by Lemma~\ref{lm-odd-cut},
$f$ cannot be adjacent to three bigons.
Let $v_1$, $v_2$ and $v_3$ be the vertices of $f$ in such an order that
there is a bigon between $v_1$ and $v_2$ and between $v_1$ and $v_3$.
Let $f'$ be the other face adjacent to the bigon between $v_1$ and $v_2$ and
$f''$ the other face adjacent to the bigon between $v_1$ and $v_3$.
Finally, let $v_4$ be the neighbor of $v_1$ on $f''$ different from $v_3$.
Also see Figure~\ref{fig-3-face-23-bigon}.

Consider an $e$-coloring as described in Lemma~\ref{lm-e-trigon} for the eligible sequence $v_1v_2v_3v_4$.
Note that the two edges of the bigon between $v_1v_3$ are colored with $\ce$ and $\cf$.
Consider a $c$-mate $M_c$ for $c\not=\ca$. This mate must contain all the edges of the bigons
between $v_1$ and $v_2$ and between $v_1$ and $v_3$.
Let $e_c$ be the edge of $f'$ contained in $M_c$. The color of $e_c$ must be $c$ and
the other face containing $e_c$ must contain another edge with the color $c$.
Consequently, it is a $\ge 4$-face. We conclude (by considering all choices of $c$) that
$f'$ is $\ge 5$-big face.
A symmetric argument applies to the face adjacent to the bigon between $v_1$ and $v_3$.
\end{proof}

Before considering $3$-faces adjacent to a single bigon,
we have to prove the following lemma:

\begin{lemma}
\label{lm-trigon-2big}
In a minimal counterexample, any trigon adjacent to an $\le 2$-big face
is also adjacent to an $\ge 4$-big face.
\end{lemma}

\begin{proof}
Let $e=v_1v_2$ be an edge of the trigon, $f$ the $\le 2$-big face adjacent to the trigon and
$f'$ the other face adjacent to the trigon.
Consider an $e$-coloring described in Lemma~\ref{lm-e-trigon} obtained for an arbitrary eligible sequence $v_1v_2v_3v_4$ and
let $M_c$ be a $c$-mate for $c\not=\ca$.

The mate $M_\cf$ contains at least five edges colored with $\cf$ by Proposition~\ref{prop-mate} and
no edges of other colors except for those contained in the trigon.
Hence, one of the (at most) two $\ge 4$-faces adjacent to $f$
shares with $f$ an edge colored $\cf$.
By symmetry with respect to the colors $\cb$, $\cc$, $\cd$ and $\ce$,
we can assume that the other $\ge 4$-face adjacent to $f$ (if it exists)
shares with $f$ an edge with a color different from $\cb$, $\cc$ and $\cd$.

Consider now a mate $M_c$, $c\in\{\cb,\cc,\cd,\cf\}$.
On the face $f$,
the mate $M_c$ either contains an edge colored with $\cf$ or (if $c\not=\cf$)
an edge colored with the color $c$ that lies on a $\le 3$-face (which
forces $M_c$ to contain an edge incident with this face that is colored with $\cf$).
In both cases,
since the mate $M_c$ contains at least five edges colored with $c$ by Proposition~\ref{prop-mate},
one of these edges (which all are colored with $c$) must lie on the face $f'$ and
the face containing this edge different from $f'$ is a $\ge 4$-face.
Since these faces are different for different $c\in\{\cb,\cc,\cd,\cf\}$,
the face $f'$ is $\ge 4$-big.
\end{proof}

We are now ready to consider $3$-faces adjacent to a single bigon.

\begin{lemma}
\label{lm-3-face-1-bigon}
If a $3$-face $f$ of a minimal counterexample is adjacent to a single bigon and
the other face adjacent to this bigon $\le 2$-big, then
the two other faces adjacent to $f$ are $\ge 3$-big.
\end{lemma}

\begin{figure}
\begin{center}
\epsfbox{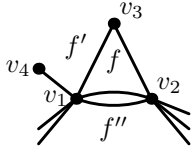}
\end{center}
\caption{Notation used in the proof of Lemma~\ref{lm-3-face-1-bigon}.}
\label{fig-3-face-1-bigon}
\end{figure}

\begin{proof}
Let $v_1$, $v_2$ and $v_3$ be the vertices of $f$ in such an order that
the bigon is between $v_1$ and $v_2$.
Let $e$ be one of the edges of the bigon between $v_1$ and $v_2$,
$f'$ the face incident with the edge $v_1v_3$,
$f''$ the other face adjacent to the bigon (which is $\le 2$-big) and
let $v_4$ be the neighbor of $v_1$ on $f'$ different from $v_3$.
Also see Figure~\ref{fig-3-face-1-bigon}.

We show that $f'$ is $\ge 3$-big. The argument for the face incident with the edge $v_2v_3$ is symmetric.
If $G$ contains a trigon between the vertices $v_1$ and $v_4$, the claim follows from Lemma~\ref{lm-trigon-2big}.
Otherwise, 
consider an $e$-coloring described in Lemma~\ref{lm-e-trigon} for the eligible sequence $v_1v_2v_3v_4$.
By swapping the colors $\ce$ and $\cf$ if necessary, we can assume that the color of the edge $v_1v_3$ is $\ce$.
Let $M_c$ be a $c$-mate for $c\not=\ca$. Observe that each mate $M_c$, $c\not=\ca$,
contains the two edges of the bigon as well as the edge $v_1v_3$, which is colored by $\ce$.

By Proposition~\ref{prop-mate}, the mate $M_\cf$ contains at least five edges colored with $\cf$.
All these edges must lie on $\ge 4$-faces (since their end-vertices are distinct).
Hence, one of the (at most) two $\ge 4$-faces adjacent to $f''$
shares an edge with color $\cf$ with $f''$.
Let $c_0$ be the color of the edge of $f''$ shared with the other $\ge 4$-face if it exists;
otherwise, let $c_0=\cf$.

On the face $f''$,
each of the mates $M_c$, $c\in\{\cb,\cc,\cd,\ce\}\setminus\{c_0\}$,
either contains an edge colored with $\cf$ or
an edge colored with the color $c$ that lies in a $\le 3$-face.
In the latter case, the other edge of that face contained in $M_c$ must have the color $\cf$.
Hence, we have exhibited the edge colored with $\cf$ of $M_c$ in both cases.
Since the mate $M_c$ contains at least five edges colored with $c$ by Proposition~\ref{prop-mate},
it contains at least three additional edges colored with $c$.
One of these edges must lie on the face $f'$ and the face containing this edge different from $f'$
must be a $\ge 4$-face. Since these faces are different for different choices of $c$,
the face $f'$ is $\ge 3$-big.
\end{proof}

We finish this subsection with an observation on faces around $3$-faces adjacent to trigons.

\begin{lemma}
\label{lm-35-face-bigon-trigon}
A minimal counterexample $G$ does not contain a vertex $v_1$ incident
with mutually adjacent $3$-face $f_3$ and a dangerous $5$-face $f_5$,
a bigon adjacent to $f_3$ but not to $f_5$, and a trigon $t$ adjacent to $f_5$.
\end{lemma}

\begin{figure}
\begin{center}
\epsfbox{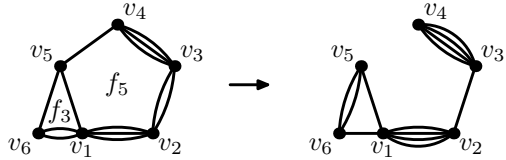}
\end{center}
\caption{Notation used in the proof of Lemma~\ref{lm-35-face-bigon-trigon}.}
\label{fig-35-face-bigon-trigon}
\end{figure}

\begin{proof}
Let $v_2$ be the other vertex contained in $t$,
let $v_3$, $v_4$ and $v_5$ be the other vertices of $f_5$ (in this order) and
let $v_6$ be the remaining vertex of $f_3$.
Also see Figure~\ref{fig-35-face-bigon-trigon}.
Note that $f_3$ is not adjacent to a trigon by Lemma~\ref{lm-3-face-trigon}.

Let $G'$ be the graph obtained from $G$ by the $v_1v_2v_3v_4v_5v_6$-swap.
By the minimality of $G$ and Lemma~\ref{lm-swap},
the graph $G'$ has a $6$-edge-coloring.
Let $C$ be the set of colors assigned to the four edges between $v_1$ and $v_2$ in $G'$.
Both the colors of the edges between $v_5$ and $v_6$ are in $C$: the two colors not contained
in $C$ are assigned to the edges $v_1v_6$ and $v_1v_5$ and thus the two edges between $v_5$ and $v_6$
cannot have either of these two colors.

At least three of the colors from $C$ are used on the edges of the quadragon
between $v_3$ and $v_4$ in $G'$.
Hence, there is a color $c$ assigned to an edge between $v_1$ and $v_2$,
an edge between $v_3$ and $v_4$, and an edge between $v_5$ and $v_6$.
Remove the three edges colored with $c$ between these three pairs of vertices and
insert the edges of $G$ missing in $G'$.
Coloring the new edges with $c$ yields a $6$-edge-coloring of $G$.
\end{proof}

\subsection{Structure of $4$-faces}

In this subsection, we prove two simple lemmas on $4$-faces.

\begin{lemma}
\label{lm-4-face-trigon}
If a trigon of a minimal counterexample $G$ is adjacent to a $4$-face $f$,
then $f$ is adjacent to no other multigon.
\end{lemma}

\begin{proof}
Let $v_1,\ldots,v_4$ be the vertices of $f$ in such an order that
the trigon adjacent to $f$ is between $v_1$ and $v_2$.
Apply the $v_1v_2v_3v_4$-swap. Since the resulting graph $G'$ contains
a new quadragon (and $f$ is not adjacent to a quadragon in $G$ by Lemma~\ref{lm-face-quadragon}),
$G'$ has a $6$-edge-coloring by the minimality of $G$ and Lemma~\ref{lm-swap}.
Let $\cb$, $\cc$, $\cd$ and $\ce$ be the colors of the edges of the quadragon in $G'$.
If one of these colors appears on the edges between $v_3$ and $v_4$,
we can use this color for the edges $v_1v_4$ and $v_2v_3$ and obtain a proper coloring of $G$.
Hence, $G'$ contains a bigon between the vertices $v_3$ and $v_4$,
the colors of its two edges are $\ca$ and $\cf$, and
$G'$ contains neither an edge $v_1v_4$ nor $v_2v_3$.
In particular, in $G$, $f$ is adjacent to a single multigon
which is the considered trigon.
\end{proof}

\begin{lemma}
\label{lm-4-face-3-bigon}
No $4$-face of a minimal counterexample is adjacent to three or four bigons.
\end{lemma}

\begin{proof}
Let $v_1,\ldots,v_4$ be the vertices of $f$ in such an order that
there are bigons (at least) between the pairs $v_1$ and $v_2$, $v_2$ and $v_3$, and $v_3$ and $v_4$.
By the minimality of $G$ and Lemma~\ref{lm-swap}, the graph $G'$ obtained by the $v_1v_2v_3v_4$-swap
has a $6$-edge-coloring. Since $G'$ contains an edge $v_2v_3$,
the trigons between the pairs $v_1$ and $v_2$, and $v_3$ and $v_4$
must have two edges with the same color.
Remove the two edges of this color from the trigons and insert them
as edges between $v_2$ and $v_3$, and $v_1$ and $v_4$. This yields a $6$-edge-coloring of $G$.
\end{proof}

\subsection{Structure of $\ge 5$-faces}

We start this subsection with two simple lemmas on $5$-faces.

\begin{lemma}
\label{lm-5-face-trigon2}
If a minimal counterexample contains a $5$-face $f=v_1v_2v_3v_4v_5$ with a trigon between $v_1$ and $v_2$,
then there is no multigon between $v_2$ and $v_3$ or there is no multigon between $v_4$ and $v_5$.
\end{lemma}

\begin{proof}
Assume that there are multigons both between $v_2$ and $v_3$ and between $v_4$ and $v_5$.
Consider an $e$-coloring described in Lemma~\ref{lm-e-trigon} for the eligible sequence $v_1v_2v_3v_5$.
The three edges between $v_1$ and $v_2$ have colors $\ca$, $\cb$, $\cc$, $\cd$ and $\ce$,
there is a bigon between $v_2$ and $v_3$ with edges colored with $\ca$ and $\cf$,
none of the edges between $v_3$ and $v_4$ has the color $\cf$, and
the edge between $v_1$ and $v_5$ is colored with $\ca$.
Since there are at least two edges between $v_4$ and $v_5$,
the mate $M_{\cf}$ does not exist.
\end{proof}

\begin{lemma}
\label{lm-5-face-multi}
In a minimal counterexample, no $5$-face is adjacent to five multigons.
\end{lemma}

\begin{proof}
Let $f=v_1v_2v_3v_4v_5$ be such a $5$-face.
By Lemma~\ref{lm-5-face-trigon2},
all the multigons adjacent to $f$ are bigons.
Consider an $e$-coloring described in Lemma~\ref{lm-e-trigon} for the eligible sequence $v_1v_2v_3v_5$.
The mate $M_{\ce}$ contains either both edges between $v_3$ and $v_4$ or between $v_4$ and $v_5$.
By symmetry, we can assume it contains the two edges between $v_3$ and $v_4$.
Consequently, these two edges are colored with $\ce$ and $\cf$,
which is impossible since one of the edges between $v_2$ and $v_3$ is also colored with $\ce$ or $\cf$.
\end{proof}

We finish this section with a lemma on the structure of $6$-faces.

\begin{lemma}
\label{lm-6-face-32}
In a minimal counterexample $G$,
no $6$-face $f$ is adjacent to three bigons and two trigons.
\end{lemma}

\begin{proof}
Let $v_1\cdots v_6$ be the vertices of the face $f$.
Lemma~\ref{lm-face-trigon} yields that we can assume
that there are bigons between $v_1$ and $v_2$, between $v_3$ and $v_4$, and between $v_5$ and $v_6$,
that there are trigons between $v_2$ and $v_3$ and between $v_4$ and $v_5$, and
that there is no multigon between $v_1$ and $v_6$.

Consider the graph $G'$ obtained from $G$ by the $v_1v_2v_3v_4v_5v_6$-swap.
By Lemma~\ref{lm-swap} and minimality of $G$, $G'$ has a $6$-edge-coloring.
Let $C_i$ be the set of three colors assigned to the edges of the trigon between $v_{2i-1}$ and $v_{2i}$ in $G'$, $i\in\{1,2,3\}$.
Since $G'$ contains a bigon between $v_2$ and $v_3$, we obtain that $|C_1\cap C_2|\ge 2$.
Analogously, we get that $|C_2\cap C_3|\ge 2$. It follows that there exists a color $c$ contained in all the three sets $C_1$, $C_2$ and $C_3$.
Removing the edges with the color $c$ from the three trigons and
coloring the edges of $G$ not present in $G'$ with $c$ yields a $6$-edge-coloring of $G$.
\end{proof}

\section{Discharging phase}
\label{sect-charge}

\subsection{Discharging rules}
\label{sub-rules}

We consider a minimal counterexample $G$ and
assign every $d$-face, $d\ge 3$, $d-3$ units of charge,
every bigon $-1$ unit of charge, every trigon $-2$ units of charge and
every quadragon $-3$ units of charge.
Vertices are assigned no charge.
This charge is referred to as {\em initial} charge.

Let us estimate the total amount of initial charge.
Since the minimal counterexample is $6$-regular,
Euler formula implies that $|F|=2m/3+2$ where $F$ is the set of faces of $G$ and $m$ is the number of its edges.
If we view a multigon of order $k$ as $k$ faces of size two,
then each $d$-face of $G$, $d\ge 2$, is assigned $d-3$ units of charge.
It follows that the initial amount of charge is equal to
$$\sum_{f\in F}\left(|f|-3\right)=2m-3|F|=-6$$
where $|f|$ stands for the size of a face $f$. In particular,
the amount of initial charge is negative.

\begin{figure}
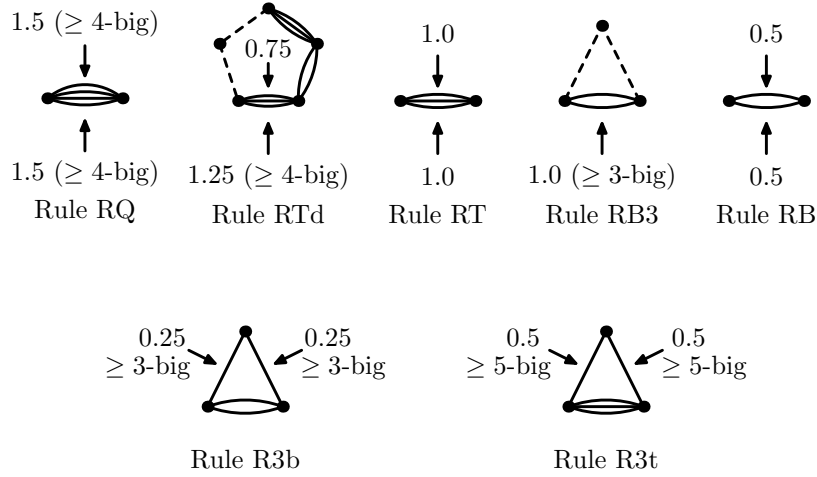

\begin{center}
\epsfbox{tjoin6.7}
\hskip 1ex
\epsfbox{tjoin6.8}
\epsfbox{tjoin6.13}
\epsfbox{tjoin6.14}
\epsfbox{tjoin6.16}
\vskip 2ex
\epsfbox{tjoin6.17}
\epsfbox{tjoin6.18}
\end{center}
\caption{Illustration of discharging rules (dashed edges denote edges that can be single or contained in multigons).}
\label{fig-rules}
\end{figure}

Next, charge gets redistributed among $\ge 3$-faces and multigons using
the following rules (also see Figure~\ref{fig-rules}).
We attempt to name the rules mnemotechnically: the names start with R, followed by
a character B, T, Q and 3 to denote the type of faces it involves (bigons,
trigons, quadragons and $3$-faces) and sometimes by another character
to distinguish the rules further (e.g., ``d'' for dangerous, ``b'' for a bigon and ``t'' for a trigon).
\begin{description}
\item[Rule RQ] Every quadragon adjacent to a $\ge 4$-big face $f$
                receives $1.5$ units of charge from $f$.
\item[Rule RTd] Every dangerous trigon receives $0.75$ units of charge from the adjacent dangerous $5$-face and
                it receives $1.25$ units of charge from the other adjacent face if that face is $\ge 4$-big.
\item[Rule RT] Every trigon that is not dangerous
               receives $1$ unit of charge from each adjacent face.
\item[Rule RB3] A bigon adjacent to a $3$-face and a $\ge 3$-big face $f$
                receives $1$ unit of charge from $f$.
\item[Rule RB] A bigon such that Rule RB3 does not apply to it
               receives $0.5$ units of charge from each adjacent face.
\item[Rule R3b] A $3$-face $f$ that is adjacent to a bigon and two $\ge 3$-big faces
               receives $0.25$ units of charge from each adjacent $\ge 3$-big face.
\item[Rule R3t] A $3$-face $f$ that is adjacent to a trigon and a $\ge 5$-big face $f'$
                receives $0.5$ units of charge from $f'$.
\end{description}
Charge of $\ge 3$-faces and multigons after these rules are applied is referred to as {\em final} charge.
In the remainder of this section, we show that final charge of every face and every multigon is non-negative.

\subsection{Final charge of multigons}

We first analyze final charge of multigons.

\begin{lemma}
\label{lm-final-multigon}
The final amount of charge of every multigon of a minimal counterexample $G$ is non-negative.
\end{lemma}

\begin{proof}
Recall that the initial amount of charge of a multigon of order $k+1$ is $-k$.
By Lemma~\ref{lm-multigon}, $G$ contains only bigons, trigons and quadragons.
Every bigon receives either $1$ unit of charge by Rule RB3 from the adjacent $\ge 3$-big face or
$0.5$ unit of charge by Rule RB from each adjacent face.
Hence, every bigon receives $1$ unit of charge in total.

Let us consider trigons.
If a trigon is not dangerous, then Rule RT applies twice.
If a trigon is dangerous,
then one of the faces adjacent to it is $\ge 4$-big by Lemma~\ref{lm-trigon-2big} (note that
the dangerous $5$-face adjacent to it is $\le 2$-big) and both parts of Rule RTd apply.
Hence, every trigon receives two units of charge in total.

Finally, every quadragon receives $1.5$ units of charge by Rule RQ
from each adjacent face since both faces adjacent to it are $\ge 4$-big by Lemma~\ref{lm-face-quadragon}.
\end{proof}

\subsection{Final charge of $3$-faces and $4$-faces}

In this subsection, we analyze final charge of $3$-faces and $4$-faces.
Let us start with $3$-faces.

\begin{lemma}
\label{lm-final-3-face}
The final amount of charge of every $3$-face $f$ of a minimal counterexample $G$ is non-negative.
\end{lemma}

\begin{proof}
Faces of a minimal counterexample send charge to adjacent multigons and $3$-faces only.
Hence, if $f$ sends out any charge, it is $\le 2$-big. In particular,
$f$ can send out some charge by Rules RT and RB only.
Consequently, if $f$ is adjacent to no multigon,
$f$ sends out and receives no charge and its final charge is zero.

If $f$ is not adjacent to multigons of order three or more,
then it is adjacent to at most two bigons by Lemma~\ref{lm-3-face-23-bigon}, and
if it is adjacent to two bigons, each of these bigons is adjacent to a $\ge 5$-big face.
Hence, Rule RB3 applies and Rule RB does not.
If $f$ is adjacent to a single bigon and the other face adjacent to the bigon is $\ge 3$-big,
again, Rule RB3 applies and $f$ sends out no charge.
If the other face adjacent to the bigon is $\le 2$-big,
then $f$ is adjacent to two $\ge 3$-big faces by Lemma~\ref{lm-3-face-1-bigon}.
In this case, $f$ sends a half of unit of charge to the bigon by Rule RB and
receives twice a quarter of unit of charge by Rule R3b.

If $f$ is adjacent to a trigon, then Lemma~\ref{lm-3-face-trigon} yields that
$f$ is adjacent to no other multigon and
the two other faces adjacent to $f$ are $\ge 5$-big.
Hence, $f$ sends the trigon $1$ unit of charge by Rule RT and
receives $0.5$ from each of the adjacent $\ge 5$-big face by Rule R3t.
The final charge of $f$ is zero.

Since $f$ cannot be adjacent to a quadragon by Lemma~\ref{lm-face-quadragon},
the proof is completed.
\end{proof}

Let us analyze final charge of $4$-faces.

\begin{lemma}
\label{lm-final-4-face}
The final amount of charge of every $4$-face $f$ of a minimal counterexample $G$ is non-negative.
\end{lemma}

\begin{proof}
By Lemma~\ref{lm-face-quadragon}, $f$ can be adjacent only to multigons of order at most three.
If $f$ is $4$-big, it sends out no charge.
If $f$ is $3$-big, only Rules RT, RB3, RB or R3b can apply and
at most one of them applies.
Hence, $f$ sends out at most one unit of charge and its final charge is non-negative.

In the rest of the proof, we assume that $f$ is $\le 2$-big
which implies that only Rules RT and RB can apply.
If $f$ is adjacent to a trigon,
then $f$ is adjacent to no other multigon by Lemma~\ref{lm-4-face-trigon}.
Hence, Rule RT applies once and no other rule can apply to $f$.
Consequently, $f$ sends out one unit of charge and its final charge is zero.

If $f$ is adjacent to no multigons of order three or more,
then, by Lemma~\ref{lm-4-face-3-bigon}, $f$ is adjacent to at most two bigons.
Hence, Rule RB can apply at most twice to $f$ and
thus the final amount of charge of $f$ is non-negative.
\end{proof}

\subsection{Final charge of $\ge 5$-faces}

In this subsection, we analyze the amount of final charge of $\ge 5$-faces.
The case of $5$-faces needs to be treated separately. So, we start with this case.

\begin{lemma}
\label{lm-final-5-face}
The final amount of charge of every $5$-face $f$ of a minimal counterexample $G$ is non-negative.
\end{lemma}

\begin{proof}
If $f$ is $\ge 4$-big, then at most one of the rules applies to $f$.
Consequently,
$f$ sends out at most $1.5$ units of charge and its final charge is positive.

If $f$ is $3$-big, then either at most two rules apply to $f$ each once or
the same rule applies to $f$ twice. Since $f$ cannot send out charge by Rule RQ and $1.25$ units of charge by Rule RTd,
$f$ sends out at most two units of charge in total and its final charge is non-negative.

If $f$ is $\le 2$-big, then only Rules RTd, RT and RB can apply to $f$.
Since no two trigons are incident by Lemma~\ref{lm-multigon},
$f$ is adjacent to at most two trigons. If $f$ is adjacent to two trigons and
no other multigons, Rule RT applies twice and the final charge of $f$ is zero.
If $f$ is adjacent to two trigons and another multigon,
then $f$ is dangerous by Lemma~\ref{lm-5-face-trigon2}.
Hence, $f$ sends twice $0.75$ units of charge by Rule RTd and once $0.5$ units of charge by Rule RB.
We conclude that the final amount of charge of $f$ is zero.

It remains to analyze the case that $f$ is adjacent to at most one trigon.
If $f$ is adjacent to a single trigon, then it is adjacent to at most two bigons by Lemma~\ref{lm-5-face-trigon2}.
Consequently, Rule RT applies once and Rule RB applies at most twice.
If $f$ is adjacent to no trigons, then it is adjacent to at most four bigons by Lemma~\ref{lm-5-face-multi} and
Rule RB applies at most four times.
In all the cases, $f$ sends out at most two units of charge and its final amount of charge is non-negative.
\end{proof}

It remains to analyze the final charge of $\ge 6$-faces. We distinguish three cases based on how big the considered face is.

\begin{lemma}
\label{lm-final-2-big}
In a minimal counterexample,
the amount of final charge of every $\le 2$-big $\ge 6$-face $f$ is non-negative.
\end{lemma}

\begin{proof}
Let $\ell\ge 6$ be the size of $f$.
Since $f$ is $\le 2$-big, it can send out charge only by Rules RT and RB.
If $f$ is adjacent to no trigon, then $f$ sends out at most $\ell/2$ units of charge.
Since the initial amount of charge of $f$ is $\ell-3\ge\ell/2$ (recall $\ell\ge 6$),
the final charge of $f$ is non-negative.

In the rest, we assume that $f$ is adjacent to a trigon.
Lemma~\ref{lm-face-trigon} implies that $f$ is adjacent to at most $\ell-1$ multigons.
Moreover, no two trigons are $f$-incident by Lemma~\ref{lm-multigon}.
Hence, Rule RT applies at most $\lfloor\ell/2\rfloor$ times and Rules RT and RB together
apply at most $\ell-1$ times. Consequently, $f$ sends out at most
\begin{equation}
\frac{1}{2}\left(\ell-1+\left\lfloor\frac{\ell}{2}\right\rfloor\right)
\label{eq-final-2-big}
\end{equation}
units of charge. The value of (\ref{eq-final-2-big}) is at most $\ell-3$ unless $\ell\in\{6,7,8\}$.
By considering the number of bigons and trigons adjacent to $f$,
we derive that $f$ sends out at most the amount of its initial charge unless one of the
following cases applies (recall that one of the edges incident to $f$ is not in a multigon and
no two trigons are $f$-incident):
\begin{itemize}
\item The face $f$ is a $6$-face and $f$ is adjacent to two trigons and three bigons.
\item The face $f$ is a $6$-face and $f$ is adjacent to three trigons and at least one bigon.
\item The face $f$ is a $7$-face and $f$ is adjacent to three trigons and three bigons.
\item The face $f$ is a $8$-face and $f$ is adjacent to four trigons and three bigons.
\end{itemize}
For every trigon $t$ adjacent to $f$,
$f$ contains an edge $e$ not $f$-incident to $t$ such that $e$ is not contained in a multigon by Lemma~\ref{lm-face-trigon}.
Since the trigons adjacent to $t$ are not $f$-incident, it follows that only the first of the four cases can possibly appear.
However, this case is excluded by Lemma~\ref{lm-6-face-32}.
\end{proof}

\begin{lemma}
\label{lm-final-3-big}
In a minimal counterexample,
the amount of final charge of every $3$-big $\ge 6$-face $f$ is non-negative.
\end{lemma}

\begin{proof}
Let $\ell\ge 6$ be the size of $f$.
Since $f$ is $3$-big, Rules RQ and RTd never apply.
It follows that $f$ sends out at most $\ell-3$ times (it does not send any charge to the three adjacent $\ge 4$-faces) at most one unit of charge.
Consequently, its final charge is non-negative.
\end{proof}

Before proving the final lemma (Lemma~\ref{lm-final-face}) of this section
which deals with $\ge 4$-big $\ge 6$-faces,
we have to state two auxiliary lemmas.
The two lemmas use the same notation
which we reuse in the proof of Lemma~\ref{lm-final-face}.

\begin{lemma}
\label{lm-final-next-A}
Let $G$ be a minimal counterexample and $f$ an $\ell$-face, $\ell\ge 6$.
Let $v_1,\ldots,v_{\ell}$ be the vertices incident with the face $f$ in the cyclic order around $f$, and
let $f_1,\ldots,f_{\ell}$ be the face or the multigon adjacent to $f$ through the edge $v_iv_{i+1}$ (indices taken modulo $\ell$).
Finally, let $s_i$ be the amount of charge sent by $f$ to $f_i$, $i=1,\ldots,\ell$.
It holds that
\begin{equation}
s_i+s_{i+1}\le 2\label{eq-final-next-A}
\end{equation}
for every $i=1,\ldots,\ell$ (indices again taken modulo $\ell$).
\end{lemma}

\begin{proof}
Assume that $s_i+s_{i+1}>2$. By symmetry, we can assume $s_i>1$, i.e., one of Rules RQ or RTd applies with respect to $f_i$.
If Rule RQ applies, then $f_i$ is a quadragon.
It follows that $f_{i+1}$ is not a multigon by Lemma~\ref{lm-multigon} and $s_{i+1}\le 0.5$ since only Rules R3b and R3t can possibly apply with respect to $f_{i+1}$.
If Rule RTd applies, then $s_{i+1}$ can be larger than $0.5$ only if Rule RB3 applies with respect to $f_{i+1}$.
However, this case is excluded by Lemma~\ref{lm-35-face-bigon-trigon}.
\end{proof}

\begin{lemma}
\label{lm-final-next-B}
Let $G$ be a minimal counterexample and $f$ an $\ell$-face, $\ell\ge 6$.
Let $v_1,\ldots,v_{\ell}$ be the vertices incident with the face $f$ in the cyclic order around $f$, and
let $f_1,\ldots,f_{\ell}$ be the face or the multigon adjacent to $f$ through the edge $v_iv_{i+1}$ (indices taken modulo $\ell$).
Finally, let $s_i$ be the amount of charge sent by $f$ to $f_i$, $i=1,\ldots,\ell$.
It holds that
\begin{equation}
s_i+s_{i+1}+s_{i+2}\le 3\label{eq-final-next-B}
\end{equation}
for every $i=1,\ldots,\ell$ (indices again taken modulo $\ell$).
\end{lemma}

\begin{proof}
If $s_i\le 1$ or $s_{i+2}\le 1$, the statement follows from Lemma~\ref{lm-final-next-A}.
So, we can assume that $s_i>1$ and $s_{i+2}>1$.
It follows that both $f_i$ and $f_{i+2}$ are multigons of order three or more and
that $s_{i+1}<1$ by Lemma~\ref{lm-final-next-A}.
Since both $f_i$ and $f_{i+2}$ are multigons of order three or more,
Rule R3t cannot apply with respect to $f_{i+1}$.
We conclude that the only two rules that can apply with respect to $f_{i+1}$ are Rules RB and R3b.

If both $f_i$ and $f_{i+2}$ are quadragons, then neither Rule RB nor Rule R3b can apply (by Lemma~\ref{lm-multigon}).
It follows $s_i=s_{i+2}=1.5$ and $s_{i+1}=0$ in this case.
If one of $f_i$ and $f_{i+2}$ is a quadragon, then $f_{i+1}$ is not a bigon and we get that $s_{i+1}\le 0.25$.
Since we have $s_i+s_{i+2}=2.75$, the statement of the lemma follows.
Finally, if both $f_i$ and $f_{i+2}$ are dangerous trigons,
we have $s_i+s_{i+2}=2.5$ and the lemma follows since $s_{i+1}\le 0.5$.
\end{proof}

We are now ready to analyze the amount of final charge of every $\ge 4$-big $\ge 6$-face.

\begin{lemma}
\label{lm-final-face}
In a minimal counterexample,
the amount of final charge of every $\ge 4$-big $\ge 6$-face $f$ is non-negative.
\end{lemma}

\begin{proof}
Let us assume that $f$ is a $k$-big $\ell$-face (note that $k\ge 4$ and $\ell\ge 6$).
Adopt the notation from the statements of Lemmas~\ref{lm-final-next-A} and \ref{lm-final-next-B}.
Further, let $i_1,\ldots,i_k$ be the indices $i$ such that $f_i$ is a $\ge 4$-face and
set $I_j=\{i_j+1,\ldots,i_{j+1}-1\}$ (indices modulo $\ell$ and $k$ where appropriate).
If $i_j+1=i_{j+1}$, then $I_j=\emptyset$.
Lemmas~\ref{lm-final-next-A} and~\ref{lm-final-next-B} imply that
\begin{equation}
\sum_{i\in I_j}s_i\le |I_j|\mbox{ for every $j=1,\ldots,k$ unless $|I_j|=1$.}\label{eq-final-face}
\end{equation}
First assume that $k\ge 6$.
Since it holds $s_i\le 1.5$ for every $i=1,\ldots,\ell$,
we obtain that
\begin{equation}
\sum_{i\in I_j}s_i\le |I_j|+0.5\mbox{ for every $j=1,\ldots,k$.}\label{eq-final-face-6+}
\end{equation}
Summing up the estimates (\ref{eq-final-face-6+}),
we get that the face $f$ sends out at most
$$s_1+\cdots+s_{\ell}\le\sum_{j=1}^k \left(|I_j|+1/2\right)=\ell-k/2\le\ell-3$$
units of charge.
In particular, its final charge is non-negative.

We now assume that $k\in\{4,5\}$.
If $|I_j|=1$, then Lemma~\ref{lm-face-quadragon} implies there is no quadragon between $v_{i_j+1}$ and $v_{i_j+2}$.
In particular, if $|I_j|=1$, then $s_{i_j+1}\le 1.25$. It follows that
\begin{equation}
\sum_{i\in I_j}s_i\le |I_j|+0.25\mbox{ for every $j=1,\ldots,k$.}\label{eq-final-face-45}
\end{equation}
Summing up the estimates (\ref{eq-final-face-45}) yields that
the face $f$ sends out at most
$$s_1+\cdots+s_{\ell}\le\sum_{j=1}^k \left(|I_j|+1/4\right)=\ell-3k/4\le\ell-3$$
units of charge.
We conclude that its final charge is non-negative.
\end{proof}

\subsection{Finale}

In order to prove Theorem~\ref{thm-edge-coloring} which implies Theorem~\ref{thm-main},
we have to exclude the existence of a minimal counterexample. Assume that $G$ is a minimal
counterexample and assign charge to the multigons and $\ge 3$-faces of $G$
as described in Subsection~\ref{sub-rules} and
apply the Rules as described.
By Lemmas~\ref{lm-final-multigon}--\ref{lm-final-face}, the final amount of charge
of every multigon and every face of $G$ is non-negative. Since charge is preserved
during the application of the rules and the sum of the amounts of initial charge
is negative, a minimal counterexample cannot exist.
This establishes Theorem~\ref{thm-edge-coloring}.

\section*{Acknowledgement}

The authors would like to thank an anonymous referee for suggesting a lemma that appears as Lemma~\ref{lm-e-trigon} in the paper;
the lemma replaced a weaker statement contained in the original draft.
This change resulted in a significant simplification of the presented arguments and shortening the paper.

\end{document}